\documentclass[12pt, a4paper]{amsart}
\usepackage{amsfonts, amsmath, amssymb, amscd, amsthm, array}
\usepackage{multicol}
\usepackage{subfig}
\usepackage{color}
\usepackage[all]{xy}
\usepackage{hyperref,url} 
\newtheorem{Thm}{Theorem}[section]
\newtheorem*{ThmA}{Theorem A}
\newtheorem*{ThmB}{Theorem B}
\newtheorem*{ThmC}{Theorem C}
\newtheorem*{ThmD}{Proposition 1}
\newtheorem{Prop}[Thm]{Proposition}

\newtheorem{Lem}[Thm]{Lemma}
\newtheorem{Cor}[Thm]{Corollary}
\theoremstyle{remark}

\theoremstyle{definition}

\newtheorem{Ex}[Thm]{Example}

\newtheorem{Not}[Thm]{Notation}
\newtheorem{Def}[Thm]{Definition}

\newcommand{\ga}{\Gamma}
\newcommand{\eps}{\varepsilon}
\newcommand{\gen}[1]{\langle\mkern3mu#1\mkern3mu\rangle}
\newcommand{\normgen}[1]{{\langle}{\mkern-3.7mu}{\lvert}{\mkern1.2mu}{#1}{\mkern1.2mu}{\rvert}{\mkern-3.7mu}{\rangle}}
\newcommand{\gp}[2]{\gen{{#1}\mid #2}}

\newcommand{\lsup}[2]{{}^{#1}\mkern-1mu{#2}}
\def\coloneq{\mathrel{\mathop\mathchar"303A}\mkern-1.2mu=}
\def\Z{\mathbb{Z}}
\def\coloneqq{\mathrel{\mathop\mathchar"303A}\mkern-1.2mu=}

\DeclareMathOperator{\Link}{Lk}

\title{One relator quotients of Graph Products}
\author{Yago Antol\'{i}n and Aditi Kar}
\address{School of Mathematics, University of Southampton, Highfield, Southampton, SO17 1BJ, England.}
\address{Mathematical Institute, University of Oxford, 24-29 St Giles',.}
\email{yago.anpi@gmail.com, Aditi.Kar@maths.ox.ac.uk}
\thanks{The first author is supported by MCI (Spain) through project MTM2011-25955 and EPSRC through project EP/H032428/1. The second author is supported by EPSRC post-doctoral fellowship EP/I020276/1.}

\begin{document}

\begin{abstract} In this paper, we generalise Magnus' Freiheitssatz and solution to the word problem for one-relator groups by considering one relator quotients of  certain classes of right-angled Artin groups and graph products of locally indicable polycyclic groups.   
\end{abstract} 


\maketitle

\section{Introduction}\label{sec:intro}

The word problem for a finitely presented group is the decidability problem of determining whether two words represent the same element. It was Max Dehn who proposed the three fundamental problems of infinite group theory, the word problem, the conjugacy problem and the isomorphism problem and proved that the word and conjugacy problems are solvable in the fundamental groups of closed surfaces of genus $\geq 2$. A remarkable theorem of Magnus' generalises Dehn's result to establish that the word problem is solvable for any group that may be defined via a presentation involving only one relator. In this paper, we generalise Magnus' solution to the word problem for one-relator groups by considering one relator quotients of right-angled Artin groups and graph products of locally indicable polycyclic groups.   

Let $\ga=(V,E)$ be a graph with vertex set $V$ and edge set $E$. The \textit{right-angled Artin group} associated to the graph $\ga$ is defined by $$A_\ga:= \langle V\  |\  [u,v]=1 \ \forall\  (u,v) \in E \rangle.$$

\noindent In the literature, right-angled Artin groups are often also called graph groups or partially commutative groups. If $\ga$ is the complete graph on $V$, then $A_\ga \cong \mathbb{Z}^n$ and the rank $n$ is equal to the cardinality $|V|$ of $V$. In the other extreme, we have the non-abelian free group of rank $|V|$, which is the right-angled Artin group associated to the totally disconnected graph with $|V|$ vertices. For a subset $U\subset V$, we  denote by $\ga_U$ the subgraph of $\ga$ spanned by $U$. Recall that the inclusion $\ga_U\to \ga$ induces an natural injective map $A_{\ga_U}\to A_\ga$. In this paper, we consider one-relator quotients of a special class of right-angled Artin groups. 

\begin{Def}\label{Def:star} A graph $L$ is called {\it starred } if it is finite and has no incidence of full subgraphs isomorphic to either $C_4$, the cycle of length 4, or $L_3$, the line of length three. A right-angled Artin group is said to be starred if its defining graph is starred. 
\end{Def}

\noindent For us, the important feature of connected starred graphs is the following: Droms showed (see \cite[Lemma]{Droms}) that every such graph has a {\it node}: a node is a vertex which is adjacent to every other vertex. 
For a starred graph we define $N=N(\ga)=\{v\in V\ga : v \text{ is a node of } \ga\}$, the set of nodal vertices. In particular $\ga_{V-N}$ is disconnected. A subset $U$ of $V$ {\it spans a sub-star} of $\ga$ if $\ga_U$ is spanned by $N$ and some connected components of $\ga_{V-N}$.

For a group $G$ and $g\in G$, we write $G/\normgen{g}$ to denote the quotient of $G$ by the normal closure of $g$ in $G$. Our main interest lies in one relator quotients of right-angled Artin groups and in this spirit we obtain the theorem below.

\begin{ThmA} Let $A_\Gamma$ be a starred right-angled Artin group and let $g$ be an element of $A_\Gamma$. Let $N$ be the set of nodal vertices of $V$ and $G\coloneqq A_\Gamma/{\normgen{g}}$. Then
\begin{enumerate}
\item[(i)] the word problem is solvable in $G$;
\item[(ii)] if $U\subset N$ and $g\not\in A_{\ga_U}$, then $A_{\ga_U}$ naturally embeds in $G$;
\item[(iii)] if $U$ spans a sub-star, and $g$ is not conjugate to an element of $A_{\ga_U}$ then $A_{\ga_U}$ naturally embeds in $G$.
\end{enumerate}
\end{ThmA} 

Right-angled Artin groups are  a special case of graph products, a more general construction. Let $\ga$ be a simplicial graph and suppose that $\mathfrak{G}=\{G_v \mid v\in V\ga\}$ is a collection of groups (called \textit{vertex groups}). The \emph{graph product} $\ga \mathfrak{G}$, of this collection of groups with respect to $\ga$, is the group obtained from the free product of the $G_v$, $v \in V\ga$, by adding the relations $$[g_v, g_u]=1  \text{ for all }  g_v\in G_v,\, g_u\in G_u \text{ such that $\{u,v\}$ is an edge of } \ga.$$

Given a graph product $G\coloneqq \ga \mathfrak{G}$ and a full subgraph $\Lambda$ of $\ga$, we denote by $G_\Lambda$ the graph product $\Lambda \{G_v\mid v\in V\Lambda\}$. It is known that the natural map $G_\Lambda\to G$ induced by the inclusion of $\Lambda$ in $\ga$ is injective.

The graph product of groups is a natural group-theoretic construction generalizing free products (when $\ga$ has no edges) and direct products (when $\ga$ is a complete graph) of groups $G_v$, $v \in V\ga$. Graph products were first introduced and studied by E. Green in her Ph.D. thesis \cite{Green}. Green considered one-relator quotients of right-angled Artin groups and asked if there exists some form of a \emph{Freiheitssatz} for these groups. Theorem A (iii) gives a Freiheitssatz for this family of groups. Green did not consider the word problem for one relator right-angled Artin groups.

Theorem A(i) is a special case of Theorem B below, in which we give a solution to the word problem for one relator quotients of graph products of poly-(infinite cyclic) groups. Recall that a group is \emph{poly-(infinite cyclic)} if there is a finite length subnormal series in which  all the factor groups are infinite cyclic. Clearly, poly-(infinite cyclic) groups are locally indicable. Recall that a group $G$ is {\it locally indicable} if every non-trivial finitely generated subgroup of $G$ has an infinite cyclic quotient.

\begin{ThmB} Let $\Gamma$ be a starred graph and $\mathfrak{G}=\{G_v\mid v\in V\ga\}$ be a family of poly-(infinite cyclic) groups.  Let $g\in G\coloneqq \ga \mathfrak{G}$. Then, the word problem  of the  one-relator quotient $G/{\normgen{g}}$ is solvable.
\end{ThmB} 

Statements (ii) and (iii) of Theorem A do not generalise to graph products of poly-(infinite cyclic) groups and to illustrate this,  an example is given in section \ref{sec:ThmB}.

In Theorem C below, we establish a \emph{Freiheitssatz} for one relator quotients of groups of the form $(A\times C)*_C(B\times C)$, where $A$ and $B$ are locally indicable. Theorem C plays a crucial role in the proof of Theorem A. Algorithmically locally indicable groups are defined in Section \ref{sec:ali}.
\begin{ThmC}
Let $A, B$ and $C$ be  groups and  $G\coloneqq (A\times C)*_C(B\times C)$. Let $w\in G$ and suppose that it is not conjugate to an element of $A\times C$ nor of $B\times C$. 
Then the following hold.
\begin{enumerate}
\item[{\rm(i)}](Freiheitssatz) If $A$ and $B$ are locally indicable, then the natural map $(A\times C)\to G/\normgen{w}$ is injective.
\item[{\rm(ii)}](Membership problem) If moreover  $A$ and $B$ are algorithmically locally indicable;
then the membership problem for $A\times C$ is solvable in the group $G/\normgen{w}.$ 
\end{enumerate}
\end{ThmC}

When $C$ is trivial, Theorem C(i)  is the local-indicability Freiheitssatz (independently proved by Brodskii \cite{Brodskii84} and Howie \cite{Howie81}) and Theorem C(ii) is a strengthening of a result of Mazurovskii \cite{Mazurovskii}. It is worth remarking here that the Freiheitssatz fails in general for direct products;
for example, in the one-relator quotient of $G=\gen{a,b\mid\;}\times \gen{c,d\mid \;}/\normgen{ac^{-1}}$, the natural map $\gen{a,b\mid \:}\to G$ is not injective.
Since $(A\times C)*_C (B \times C)$ $ \cong (A*B) \times C$, Theorem C(i) also provides a Freiheitssatz for direct products as follows. 

\begin{Cor}
Let $A$ and $B$ be two locally indicable groups and $C$ any group. If  $g\in (A*B)\times C$ is not conjugate to an element of $A\times C$ or $B\times C$ then $C$ naturally embeds in the one-relator quotient  $((A*B)\times C)/\normgen{g}$.
\end{Cor}

\noindent The word problem for one relator quotients of direct product of groups is much simpler.
\begin{ThmD} Let $A$ and $B$ be two recursively presented groups and $a\in A$ and $b\in B$ such that the word problem is solvable in $A$, $A/\normgen{a}$, $B$ and $B/\normgen{b}$. Then the word problem is solvable in  $G=(A \times B)/\normgen{(a,b)}$. 
\end{ThmD} 

Proposition 1 greatly enlarges the class of one relator graph products with solvable word problem and combined with Theorem B, Proposition 1 provides the Corollary below. 
\begin{Cor}
Let $A$ and $B$ be two starred graph products of poly-(infinite cyclic) groups. Then, for all $g\in A\times B$, $(A\times B)/\normgen{g}$ has solvable word problem.
\end{Cor}

\begin{Not}
Let $x,y$ be elements of a group $G$. We denote the commutator of $x$ and $y$ by  $[x,y]\coloneqq xyx^{-1}y^{-1}$ and $x^y\coloneq y^{-1}xy$. Similarly, for $X,Y\subset G$,  $X^Y\coloneqq \{ x^y :x\in X, y\in Y\}.$
\end{Not}

The paper is organized as follows:
In section \ref{sec:ali} we introduce the concept of algorithmically locally indicable groups and show that chordal graph products of poly-(infinite cyclic) groups have this property. In section \ref{sec:ThmC} we prove Theorem C, in section \ref{sec:DP} we prove Proposition 1 and finally, Theorems A and B  in section  \ref{sec:ThmB}.

\section{Algorithmic Local Indicability}\label{sec:ali}
 
Higman called a group $G$ locally indicable if every non-trivial finitely generated subgroup of $G$ has an infinite cyclic quotient. As we wish to employ local indicability for solving word problems, we require a computable form of Higman's notion. In this section, we define the notion of algorithmic local indicability which plays a crucial role in Theorems A, B and C. We show that a graph product of poly-(infinite cyclic) groups is algorithmically locally indicable.

We say that \emph{the uniform membership problem is solvable in a group $G$} (also known as \emph{generalized word problem}) if there exists an algorithm which, given any finitely generated subgroup $H\leq G$ and an element $g \in G$, decides whether $g$ belongs to $H$ or not. 

\begin{Def}\label{ALI} A group $G$ is said to be algorithmically locally indicable if
\begin{enumerate}
\item the uniform membership problem is solvable in $G$, and 
\item \label{map} there exists a \textit{computable} surjection from every finitely generated subgroup of $G$ to the integers, i.e. there is an algorithm which, given any finite subset $\{w_1,\ldots, w_k\}\subset G$ that generates a non-trivial subgroup $H$ in $G$, produces a list of integers $\{n_1, \ldots, n_k\}$ with the property that mapping each $w_i$ to the corresponding $n_i$ defines an epimorphism of $H$ onto $\mathbb{Z}$. 
\end{enumerate} 
\end{Def}

\subsection*{Examples. } 1.) Abelian and non-abelian free groups are algorithmically locally indicable. The abelian case follows easily from the fundamental structure theorem for finitely generated abelian groups. If $F$ is a non-abelian free group, then we know the uniform membership problem is solvable in $F$ from \cite{LyndonSchupp}. Now let $g_1, \ldots,g_n \in F$ and let $H=\langle g_1, \ldots,g_n \rangle$ be the subgroup that they generate in $F$. Perform Nielsen transformations to obtain a (Nielsen reduced) basis for $H$. Now there are many ways to define an epimorphism from $H$ onto $\mathbb{Z}$!

2.) The right-angled Artin group associated to $C_4$ is a direct product of two non-abelian free groups of rank 2. Mikhailova showed in \cite{Mikhailova} that the uniform membership problem is not solvable in $F_2 \times F_2$. Hence this group is locally indicable but is not algorithmically locally indicable.  
\medskip

A graph $L$ is chordal if it is finite and every cycle of length at least 4 has a chord, i.e. an edge joining two non-adjacent vertices of the cycle. Notice that every starred graph is chordal.

\begin{Prop}\label{prop:chordal} Let $\ga$ be a chordal graph, $\mathfrak{G}=\{G_v\mid v\in V\ga\}$ a family of poly-(infinite cyclic) groups. Then $G\coloneqq\ga \mathfrak{G}$ is algorithmically locally indicable. 
\end{Prop}

We remark that chordal graph products of poly-(infinite cyclic) groups is the larger family of graph products for which the uniform membership problem is solvable. In the particular case of starred right-angled Artin groups, this was implicit in the proof of Droms \cite[Theorem]{Droms} when he proves that all subgroups of starred right-angled Artin group are starred.

To prove the proposition above, we need a few lemmas. 

\begin{Def}
A short exact sequence of groups
$1\to K\stackrel{\phi}{\to} G\stackrel{\psi}{\to} H\to 1$
is said to be {\it algorithmic} if:
\begin{enumerate}
\item[(i)] the groups $K$, $G$ and $H$ and the
morphisms $\phi$ and $\psi$ are given to us in an algorithmic way; i.e. we can effectively
operate in $K,$ $G$ and $H$, and compute images under $\phi$ and $\psi$, 
\item[(ii)] we can solve the word problem in $K$ and $H$,
\item[(iii)]  given $g\in G$, such that $\psi(g)=1$ we can find $k\in K$ such that $\phi(k)=g$. 
\end{enumerate}
\end{Def}

\begin{Lem} \label{lem:extconst}
Let $K, H$ be two groups which satisfy property (\ref{map}) of Definition \ref{ALI}. Let $1\to K\stackrel{\phi}{\to} G\stackrel{\psi}{\to} H\to 1$ be an algorithmic short exact sequence. Then $G$ satisfies property (\ref{map}) of Definition \ref{ALI}. 
\end{Lem}
\begin{proof}
Given $\{g_1,\dots, g_n\}$ in $G$ we can compute its image in $H$ under $\psi$. Since the word problem in $H$ is solvable we can decide whether or not $\gen{\psi(g_1),\dots, \psi(g_n)}$ is the trivial group. If $\{\psi(g_1),\dots, \psi(g_n)\}$ generates a non-trivial subgroup, then the algorithm of $H$ produces a computable epimorphism  from $\gen{\psi(g_1),\dots, \psi(g_n)}$ to $\Z$. Composing with $\psi$ we obtain our computable epimorphism from $\gen{g_1,\dots, g_n}$ to $\Z$.

If $\gen{\psi(g_1),\dots, \psi(g_n)}$ is trivial, then $\psi(g_i)=1$ for each $i=1,\dots, n$ and for each $i$, we can find a unique $k_i\in K$ such that $\phi(k_i)=g_i$. Observe that the injective map $\phi$ restricts to an isomorphism of $\gen{k_1,\dots, k_n}$ with the subgroup of $G$ generated by the $g_i$. Since the word problem in $K$ is solvable, we can decide if all of the $k_i$ are trivial or not. If at least one of the $k_i$'s is non-trivial in $K$, then the algorithm in $K$ produces a computable epimorphism from $\gen{k_1,\dots, k_n}$ to $\Z$ which furnishes us with the required epimorphism. 
\end{proof}

\noindent We will repeatedly use the following lemma in the rest of the paper. 
\begin{Lem} \label{lem:free_prod} The free product of two algorithmically locally indicable groups is algorithmically locally indicable. 
\end{Lem}

\begin{proof} Let $A$ and $B$ be algorithmically locally indicable groups and let $S$ be a nonempty finite set of words in $A*B$. Mikhailova provides a solution to the uniform membership problem for free products in \cite{Mikhailova68} and so we need to verify property (\ref{map}) of Definition \ref{ALI}. First we write each element of $A*B$ in \emph{normal form}: every element $w\neq 1$ of the free product $A*B$ can be uniquely expressed as a product $w=g_1\ldots g_k$ such that each $g_i \neq 1$, each $g_i$ belongs to one of the factors $A$ or $B$ and successive $g_i$, $g_{i+1}$ are not in the same factor. In particular, we can effectively compute the image of every $g\in G$ in $A\times B$ under the natural quotient map $\pi\colon A*B\to A\times B$. 

Recall that in each factor one may decide if a given word is trivial and then we can decide if a word is trivial in $A\times B$.  Let $K=\ker \pi$ and recall that \cite[Prop. I.1.4]{serre} says that $K$ is free with basis $\{[a,b]^{\pm 1}, a\in A, b\in B\}$. Therefore $K$ is algorithmically locally indicable. Denote by $\phi$ the natural inclusion of $K\to A*B$. 
%

Note that $a_1b_1a_2b_2$, $a_i \in A$, $b_i \in B$ can be rewritten as $a_1a_2b_1b_2[(b_1 b_2)^{-1},a_2^{-1}][a_2^{-1},b_2^{-1}]$ and one can further prove using induction that $a_1b_1\ldots a_kb_k$ is precisely
$$a_1\ldots a_k b_1\ldots b_k\left(\prod_{i=1}^{k-1}{[(b_i\ldots b_{k})^{-1}, (a_{i+1} \ldots a_{k})^{-1}][(a_{i+1}\ldots a_{k})^{-1}, (b_{i+1}\ldots b_{k})^{-1}] } \right)$$


Therefore any element $a_1b_1\dots a_kb_k$ of $A*B$ that lies in the kernel $K$ (i.e. $a_1\dots a_k=1$ and $b_1\dots b_k=1$) may be rewritten as a product of commutators of the form $[a,b]^{\pm 1}, a\in A, b\in B$. That is, for every $g\in A*B$, $\pi(g)=1$, we can compute a $k\in K$ (as a word in the free basis of $K$) such that $\phi(k)=k$. 

In other words, we have proved that $1\to K\stackrel{\phi}{\to}A*B\stackrel{\pi}{\to}A\times B\to 1$ is an algorithmic short exact sequence and now the lemma follows from  Lemma \ref{lem:extconst}  that $A*B$ satisfy the condition (\ref{map}) of Definition \ref{ALI}.

 \end{proof}

\begin{Cor}\label{cor:SP}
 Poly-(infinite cyclic) groups are algorithmically locally indicable. 
\end{Cor}

\begin{proof}
The uniform membership problem is shown to be solvable for polycyclic groups in \cite{BCRS}. To verify that condition (2) of Definition \ref{ALI} holds, one proceeds by induction on the Hirsch length of the polycyclic group and repeatedly applies Lemma \ref{lem:extconst}.
\end{proof}

\begin{Lem}\label{cond2}
Let $\ga$ be a simplicial graph and  $\mathfrak{G}=\{G_v \mid v\in V\ga\}$ a family of groups that satisfy property \eqref{map} of Defintion \ref{ALI}. Then $G\coloneqq \ga \mathfrak{G}$ satisfies property \eqref{map} of Defintion \ref{ALI}.  
\end{Lem}

\begin{proof} Let $S= \{w_1,\ldots,w_k\} \neq \{1\}$ be a finite collection of words from $G$ and let $H$ be the subgroup generated by $S$ in $G$. The general case reduces to the situation when $\Gamma$ is finite. This is because for infinite $\Gamma$, the finite set $S$ involves a finite subset $\Lambda$ of $\Gamma$. Moreover, $G_\Lambda $ is a retract of $G$. Once we have a computable epimorphism of the subgroup generated by $S$ in $G_\Lambda$, composing with the canonical surjection $G\rightarrow G_\Lambda$, we get the required computable epimorphism from $H$ onto $\Z$. Therefore we may assume that the defining graph $\Gamma$ is finite. 

We now proceed by induction on the size of its vertex set $|V|$.  If $|V|=1$ then $G \cong G_v$, $v\in V$ and there is nothing to prove. Fix $v \in V$ and consider $\Lambda$, the full subgraph of $\Gamma$ with vertex set $V - \{v\}$. Also, the normaliser of $G_v$ in $G$ is precisely $G_v \times G_{\Link(v)}$, here $\Link(v)$ is the {\it link} of $v$ in $\Gamma$, that is the full subgraph of $\Gamma$ whose vertex set are vertices adjacent to $v$. It is now easy to see that $G \cong (G_v \times G_{\Link(v)} )*_{G_{\Link(v)}}  G_\Lambda$ and using Bass-Serre theory, we deduce that $\gen{G_v^G}$, the normal subgroup generated by $G_v$, is a free product of $*_{t\in T}G_v^t$, where $T$ is in one-one correspondence with cosets of $G_v\times G_{\Link(v)}$ in $G$. Our graph product $G$ fits into the split short exact sequence: 
\begin{equation}\label{eq:seq}1 \rightarrow \gen{ G_v^{G}}  \stackrel{i}{\rightarrow}G \stackrel{\pi}{\rightarrow} G_\Lambda \rightarrow 1.\end{equation}

Using normal forms for graph products, we can solve the membership problem for $G_v\times G_{\Link(v)}$ in $G$. Hence, one can easily choose minimal length representatives for cosets of $G_v\times G_{\Link(v)}$ in $G$. Then we can compute the normal forms in the free product decomposition of $\gen{G_v^G}$ of pre-images of elements of $G$ mapping to $1$ in $G_\Lambda$ and,  by Lemma \ref{lem:free_prod}, $\gen{ G_v^{G}}$ satisfies property \eqref{map} of Definition \ref{ALI}. 

Note that, using normal forms, $i$ and $\pi$ in \eqref{eq:seq} are given in an algorithmic way. Hence \eqref{eq:seq} is an algorithmic short exact sequence. Moreover, by induction hypothesis, property \eqref{map} of Definition \ref{ALI} holds for  $G_\Lambda$. Therefore, the lemma follows from  Lemma \ref{lem:extconst}.  
\end{proof}

\begin{proof}[Proof of Proposition \ref{prop:chordal}] Let $G$ be graph product of poly-(infinite cyclic) groups over a chordal graph. Arguing as in \cite[Corollary 1.3]{Kapovich}, one can prove that $G$ is the fundamental group of a graph of groups $(Y,G(-))$, where $Y$ is a tree and the vertex and edge groups are polycyclic and so by \cite[Theorem 1.1]{Kapovich} the uniform membership problem is solvable. We need to verify property (\ref{map}) of Definition \ref{ALI}. Since poly-(infinite cyclic) groups are algorithmically locally indicable (Corollary \ref{cor:SP}), the proposition follows from Lemma \ref{cond2}. 
\end{proof}

\section{Proof of Theorem C}\label{sec:ThmC}
The proof of Theorem C is similar to the one for free products and is based on \cite{BBaumslag84, Brodskii84}. We will repeatedly use the following fact.
\begin{Lem}
Let $G=(A\times C)*_C(B\times C)$. Then $G\cong (A*B)\times C$ and any automorphism of $A*B$ extends to an automorphism of $G$ that fixes $C$.
\end{Lem}

\begin{Lem}\label{lem:kerphi}
Let $A$ be generated by $(a_1,\dots, a_k)$ and let $(z_1,\dots,z_k)\in \Z$ such that $a_i\mapsto z_i$, $i=1,\dots,k$ extends to an epimorphism $\phi\colon A\to \Z$. Let $a\in A$ such that $\phi(a)=1$. Let $n>0$ be an integer  and $\widehat{A}\coloneqq A*_{\gen{ a= \alpha^n}}\gen{\alpha |\;}$, i.e. the group obtained from $A$ by adding an $n$-th root of $a$.
\begin{itemize}
\item[\textrm{(i)}] If $A$ is  locally indicable, so is $\widehat{A}$.
\item[\textrm{(ii)}] If $A$ is algorithmically locally indicable, so is $\ker \widehat{\phi}$, where $\widehat{\phi} \colon \widehat{A}\to \Z$ is  defined by $\widehat{\phi}(g)=n\cdot g$, $g\in A$ and $\widehat{\phi}(\alpha)=1.$
\end{itemize}
\end{Lem}
\begin{proof}

Let $K\coloneqq \ker \phi$ and $\widehat{K}\coloneqq \ker \widehat{\phi}$. The group $\widehat{K}$ acts on the Bass-Serre tree $T$ of $A*_{\gen{ a= \alpha^n}}\gen{\alpha |\;}$. The intersection of $\widehat{K}$ with the conjugates of $\gen{\alpha |\;}$ is trivial and the intersection of $\widehat{K}$ with conjugates of $A$ is isomorphic to $K$. Then using Bass-Serre theory \begin{equation}\label{eq:freeprod}\widehat{K}= *_{g\in \widehat{K}\backslash \widehat{A} /A}(\widehat{K}\cap gAg^{-1})*F\end{equation}  where $F$ is a free group. Since $\widehat{K}\backslash \widehat{A}\cong \gen{\alpha}$ and $\alpha^n\in A$, we conclude that $\{1,\alpha,\dots, \alpha^{n-1}\}$ is a transversal for the double cosets $\widehat{K}\backslash \widehat{A} /A$. Moreover, since $\widehat{K}=\gen{K,\alpha K \alpha ^{-1},\dots, \alpha^{n-1}K \alpha^{n-1}}$, we conclude that $F=1$.

If $A$ is locally indicable, $K$ is locally indicable and $\widehat{K}$ is locally indicable. Thus, $\widehat{A}$ is an extension of locally indicable groups and hence locally indicable. This proves (i).

To prove (ii) observe that for any element $k \in \widehat{K}$, we can effectively compute the normal form of $k$ in the free product  \eqref{eq:freeprod}. By Lemma \ref{lem:free_prod}, $\widehat{K}$ is algorithmically locally indicable since it is a free product of algorithmically locally indicable groups.
\end{proof}

\begin{Lem}\label{lem:memberZ}
Let $A$ be a group and let $a\in A$. Suppose that the membership problem for $\gen{a}$ in $A$ is solvable. Then, for any group $C$, the membership problem for $\gen{a, C}$ is solvable in $A\times C$.
\end{Lem}\begin{proof}
The subgroup $\gen{a,C}$ is isomorphic to $\gen{a}\times C$ and so an element $(a_0,c_0)\in A\times C$ belongs to $\gen{a,C}$ if and only if $a_0$ belongs to $\gen{a}$, 
\end{proof}

\begin{ThmC}
Let $A, B$ and $C$ be  groups and  $G\coloneqq (A\times C)*_C(B\times C)$. Let $w\in G$ and suppose that it is not conjugate to an element of $A\times C$ nor of $B\times C$. 
Then the following hold.
\begin{enumerate}
\item[{\rm(i)}](Freiheitssatz) If $A$ and $B$ are locally indicable, then the natural map $(A\times C)\to G/\normgen{w}$ is injective.
\item[{\rm(ii)}](Membership problem) If moreover  $A$ and $B$ are algorithmically locally indicable;
then the membership problem for $A\times C$ is solvable in the group $G/\normgen{w}.$ 
\end{enumerate}
\end{ThmC}

\begin{proof}

Since $G\cong (A*B)\times C$, we can use the normal form of free products to express the element $w$  as a reduced word $a_1b_1\dots a_lb_lc$, where $a_i\in A,$  $b_i\in B$ and $c\in C$.  After conjugation, if necessary, we can assume that $a_i\neq 1\neq b_i$ for $i=1,\dots, l$. The proof proceeds by induction on $l$.

If $l=1$, since $w$ is not conjugate to an element of either $A\times C$ nor of $B\times C$, both $a_1, b_1$ are non-trivial. Moreover, $A$ and $B$ are locally indicable which implies that $\gen{a_1}\cong\gen{b_1}\cong\Z$ and $\gen{C,a_1}\cong \gen{C,b_1c}\cong C\times \Z$. The group $G/\normgen{w}$ is therefore the amalgamated free product $(A\times C)*_{\gen{C,a_1}=\gen{C,b_1c}}(B\times C)$ and (i) holds. Under the hypothesis of (ii), the membership problem $\gen{a_1}$ in $A$ is solvable (resp. $\gen{b_1}$ in $B$) and so by Lemma \ref{lem:memberZ}, the membership problem for $\gen{a_1, C}$ is solvable in $A\times C$, (resp. $\gen{b_1c, C}$ in $B\times C$). Hence, normal forms may be effectively computed in $G/\normgen{w}$ view as an amalgamated free product and then, the membership problem for $A\times C$ (and $B\times C$) is solvable in $G/\normgen{w}$ and (ii) holds. 

Suppose now that $l$ is at least 2. Let $A_0=\gen{a_1,\dots a_l},$  $B_0=\gen{b_1,\dots b_l}$ and $G_0=(A_0\times C)*_C (B_0\times C)=\gen{A_0,B_0}\leqslant G$. If $A$ and $B$ are algorithmically locally indicable, so are $A_0$ and $B_0$. To show (i), it is enough to show that $(A_0\times C)$ naturally embeds in $G_0/\normgen{w}$, since then $A\times C$ naturally  embeds in $(A\times C)*_{(A_0\times C)} (G_0/\normgen{w})$. To show (ii), assuming that (i) holds, it is enough to show that $(A_0\times C)$ and $(B_0\times C)$ have  solvable membership problem in $G_0/\normgen{w}$, since then we can compute normal forms in $$G/\normgen{w}=(A\times C)*_{(A_0\times C)} (G_0/\normgen{w})*_{(B_0\times C)} (B\times C)$$ and in particular solve the membership problem for $A\times C$ and $B\times C$ in $G/\normgen{w}$. So we can assume from henceforth that $A=A_0,$ $B=B_0$ and $G=G_0.$ 

Since $A$ and $B$ are locally indicable, there exists epimorphism $\phi_A\colon A\to \Z$ and $\phi_B \colon B\to\Z$, and, since $G\cong (A*B)\times C$, we can extend them to $\widehat{\phi_A}\colon G\to \Z$ and $\widehat{\phi_B}\colon G \to \Z$.  For the proof of (ii) we can assume that $\phi_A,$ $\phi_B,$ $\widehat{\phi_A}$ and $\widehat{\phi_B}$  are computable. 

\textbf{Case 1:} $\widehat{\phi_A}(w)=0$ or $\widehat{\phi_B}(w)=0.$
In this case we will assume that $\widehat{\phi_A}(w)=0.$ Due to the asymmetry of this case, to cover all possibilities, we have to show that $A\times C$ and $B\times C$ naturally embed in $G/\normgen{w}$ and that, under the extra hypothesis  of (ii), $(A\times C)$ and $(B\times C)$ have solvable membership problem in $G/\normgen{w}.$

Choose $a\in A$  such that $\widehat{\phi_A}(a)=1$. Let $\widetilde{A}$ denote the kernel of $\phi_A$ and let $H$ denote the kernel of  $\widehat{\phi_A}$. Then, writing $B^i\coloneqq a^{-i}Ba^{i}$, we have 
$$H=C\times (\widetilde{A}*B^0*B^{-1}*B^{1}*B^{-2}*B^{2}*\dots )$$
For $p,q\in \Z$, $p\leq q$ let  $$H_{[p\uparrow q]}=\widetilde{A}*B^{p}*B^{p+1}*\cdots *B^{q}\leqslant H.$$
For $i=1,\dots, l,$ let $m_i=\phi_A(a_i)$ and we rewrite $a_i$ as  $a^{m_i}\tilde{a}_i,$ where $\tilde{a}_i\in \widetilde{A}$. Since $\phi_A(w)=0$, we can rewrite $w$ as 
\begin{equation}\label{eq:wrewritten}
a^{m_1}\tilde{a}_1a^{-m_1} (a^{m_1} b_1 a^{-m_1}) a^{m_1+m_2}\tilde{a}_2 \cdots (a^{m_1+\dots+m_l}b_l a^{-m_1-\dots-m_l})c
\end{equation}

 Let $\mu=\min\{ \sum_{i=1}^j -m_i: j=1,\dots, l \}$ and $\nu= \max\{ \sum_{i=1}^j -m_i: j=1,\dots, l \}$. After conjugating $w$ by $a$ or $a^{-1}$ we can assume that $\mu\leq 0$. If $\nu=\mu$ then, since $\sum_{i=1}^l m_i=0$, we have that $m_i=0$ for $i=1,\dots, l$, and hence $\phi_A$ is the constant map $0$, a contradiction. Therefore, $\mu<\nu$, $w\in H_{[\mu\uparrow \nu]}\times C$, and using normal forms,  $w$ is not conjugate to an element of $H_{[\mu+1\uparrow \nu]}\times C$ nor of $H_{[\mu\uparrow \nu-1]}\times C$. Moreover, since the sequence $(\sum_{i=1}^j -m_i: j=1,\dots, l )$ is not constant, the number of times that it takes the value $\mu$ (resp. $\nu$) is strictly less than $l$. Then $w$ as a word in the free product with amalgamation $(H_{[\mu\uparrow \nu-1]}\times C)*_C(B^{\nu}\times C)$ and as a word in $(H_{[\mu+1\uparrow \nu]}\times C)*_C(B^{\mu}\times C)$ has length less than $l$.

Since the free product of (algorithmically) locally indicable groups is (algorithmically) locally indicable,  $H_{[\mu+1 \uparrow \nu]}$, $H_{[\mu\uparrow \nu-1]}$, $B^{\mu}$ and $B^{\nu}$ are (algorithmically) locally indicable. Now, the induction hypothesis says that $H_{[\mu+1\uparrow \nu]}\times C$ and $B^{\mu}\times C$ naturally embed (resp. have solvable membership problem) in $$(H_{[\mu\uparrow \nu]}\times C)/\normgen{w}=((H_{[\mu+1\uparrow \nu]}\times C)*_{C} (B^\mu\times C))/\normgen{w},$$ and  $H_{[\mu\uparrow \nu-1]}\times C$ and $B^{\nu}\times C$ naturally embed (resp. have solvable membership problem) in $$(H_{[\mu\uparrow \nu]}\times C)/\normgen{w}=((H_{[\mu\uparrow \nu-1]}\times C)*_{C} (B^\nu\times C))/\normgen{w}.$$

Assume that $A$ and $B$ are locally indicable. The conjugation by $a$ in $H$ induces an isomorphism  $a\colon (H_{[\mu\uparrow\nu-1]}\times C) \to (H_{[\mu+1\uparrow\nu]}\times C)$ and $G/\normgen{w}$ is the HNN-extension 
\begin{equation}\label{eq:HNNext}
G/\normgen{w}=\left\langle\ \frac{H_{[\mu\uparrow \nu]}\times C}{\normgen{w}}, a\ \vline\ a(H_{[\mu\uparrow\nu-1]}\times C)a^{-1}= H_{[\mu+1\uparrow\nu]}\times C \right\rangle
\end{equation}
Then $(H_{[\mu\uparrow \nu]}\times C)/\normgen{w}$ embeds in $G/\normgen{w}$. In particular  $B^0\times C$ embeds in $(H_{[\mu\uparrow \nu]}\times C)/\normgen{w}$, that is $B\times C$ embeds in $G/\normgen{w}$. Similarly $\widetilde{A}\times C$ naturally embeds in $G/\normgen{w}$. 
To show that $A\times C$ naturally embed in $G/\normgen{w}$ suppose that $x,y\in A \times C$ such that $x\in y\gen{w^G}$ (i.e. $x$ and $y$ are equal in $G/\normgen{w}$), we have to show that $x=y$. Since $\widehat{\phi_A}(w)=0$, the map $\widehat{\phi_A}\colon G\to \Z$ factors through $G/\normgen{w}$.  Hence, $x,y$ have the same image under $\widehat{\phi_{A}}$ and then $y^{-1}x\in(\tilde{A}\times C)\cap \gen{w^{G}}$. Since $\tilde{A}\times C$ naturally embed in $G/\normgen{w}$ the previous intersection is the trivial group and then $x=y$. This completes the proof of (i) in Case 1.

Assume now that $A$ and $B$ are algorithmically locally indicable. Then, as noticed before, the membership problem for $H_{[\mu+1\uparrow \nu]}\times C$ and $H_{[\mu\uparrow \nu-1]}\times C$  in $(H_{[\mu\uparrow \nu]}\times C)/\normgen{w}$ is solvable. This means that normal forms in the HNN-extension \eqref{eq:HNNext} can be effectively computed and we can solve the membership problem for $\tilde{A}\times C$ and $B\times C$ in $G/\normgen{w}$. To solve the membership problem for $A\times C$ in $G/\normgen{w}$, take $x\in G$ and $y\in A\times C$ such that $\widehat{\phi_A}(x)=\widehat{\phi_A}(y)$. Taking into account that $\widehat{\phi}_A$ is constructible, given $x\in G$ we can compute such an $y$. Notice that $x\in A\times C$ if and only if $xy^{-1}\in \tilde{A}\times C$.  It follows that  we can solve the membership problem for $A\times C$  in $G/\normgen{w}$. This completes the proof of (ii) in Case 1.

\textbf{Case 2:} $\widehat{\phi_A}(w)=m\neq 0$ and $\widehat{\phi_B}(w)=n\neq 0.$
Due to the symmetry of this case, we only have to show that $A\times C$ embeds in $G/\normgen{w}$ and that, under the extra hypothesis of (ii), $A\times C$ has solvable membership problem in $G/\normgen{w}.$ 

\textbf{Subcase 2.1:} $B\cong \Z.$

In this situation $B=\gen{b}$ and $w=a_1 b^{n_1}\cdots a_lb^{n_l}c$. It is important to notice that no $n_i$ is equal to 0. Let  $n=\phi_B(w)=n_1+\dots +n_l.$

Let $\widehat{A}$ be the group obtained from $A$ by adding $\alpha$, an $n$th-root of $a$, that is  $\widehat{A}=A*_{\gen{a=\alpha^n}}\gen{\alpha| \;}$. We construct an epimorphism $\phi_{\widehat{A}}\colon\widehat{A}\to \Z$  by setting $\phi_{\widehat{A}}(g)=n\cdot \phi_A(g)$, for $g\in A$ and  $\phi_{\widehat{A}}(\alpha)=1$.  Let $\widehat{G}=(\widehat{A}\times C)*_C(B\times C)=(\widehat{A}*B)\times C.$ We can extend $\phi_{\widehat{A}}$ to $\widehat{\phi_{\widehat{A}}}\colon \widehat{G}\to \Z.$ Under the hypothesis of (ii), both $\phi_{\widehat{A}}$ and $\widehat{\phi_{\widehat{A}}}$ are computable epimorphisms. If $m=\widehat{\phi_A}(w)$, we have that $\widehat{\phi_{\widehat{A}}}(w)=nm$.

We define a map $f\colon \widehat{A}\cup B\cup C\to \widehat{A}\cup B\cup C$ by $f(g)=g$ for all $g\in \widehat{A}\cup C$ and $f(b)=b\alpha^{-m}$. Clearly $f$ extends to  an isomorphism of $\widehat{G}$ (also denoted by $f$). Then $f(w)=a_1(b\alpha^{-m})^{n_1}\dots a_l(b\alpha^{-m})^{n_l}c$ and $\widehat{\phi_{\widehat{A}}}(f(w))=0$.

Note that, by Lemma \ref{lem:kerphi}, $\widehat{A}$ is still locally indicable, and, if $A$ is algorithmically locally indicable, so is $\widetilde{A}=\ker \phi_{\widehat{A}}$. Let $H$ denote the kernel of  $\widehat{\phi_{\widehat{A}}}$, then  
$$H=C\times (\widetilde{A}*B^0*B^{-1}*B^{1}*B^{-2}*B^{2}*\dots )$$
where $B^i\coloneqq \alpha^{-1}B\alpha^{i}$. For $p,q\in \Z$, $p\leq q$ let $H_{[p\uparrow q]}$ denote $\widetilde{A}*B^{p}*B^{p+1}*\cdots *B^{q}\leqslant H$. We record the following
\begin{equation}\label{eq:Hloq}
\text{If $A$ and $B$ are algorithmically indicable then so is $H_{[p\uparrow q]}$.}
\end{equation}
Let $m_i=\phi_{\widehat{A}}(a_i)$ then $a_i=\alpha^{m_i}\tilde{a}_i$. To simplify notation, we will write $\lsup{k}b$ to denote $\alpha^{k}b\alpha^{-k}$.
We can rewrite $f(w)$ as
\begin{align*}
f(w)&=\alpha^{m_1}\tilde{a}_1(b\alpha^{-m})^{n_1}\dots \alpha^{m_l}\tilde{a}_l(b\alpha^{-m})^{n_l}c\\
 &=a'_1(\lsup{k_{1,1}}{b})^{\eps_{1,1}}\cdots (\lsup{k_{1,n_1}}{b})^{\eps_{1,n_1}}
\dots a'_l(\lsup{k_{l,1}}{b})^{\eps_{l,1}}\cdots (\lsup{k_{l,n_l}}{b})^{\eps_{l,n_l}}c
\end{align*}
where $a'_i=\alpha^{k_{i,1}}\tilde{a_i}\alpha^{-k_{i,1}}\in \widetilde{A}$, $\eps_{i,j}=n_i/|n_i|\in \pm 1,$  $1\leq i\leq l$ and $k_{i,j}=\left(\sum_{p=1}^{i} m_p\right) -\left(\sum_{p=1}^{i-1} n_i m\right)-(j-1)n$, if $n_i$ is positive and $k_{i,j}=\left(\sum_{p=1}^{i} m_p\right) -\left(\sum_{p=1}^{i-1} n_i m\right)-jm$, if $n_i$ is negative.

 Let $\mu=\min \{ k_{i,j} :1\leq i\leq l, 1\leq j\leq n_{i}\}$ and $\nu=\max\{k_{i,j} : 1\leq i\leq l, 1\leq j\leq n_i\}.$  Since for all $i=1,\dots, l$, $n_i\neq 0$ we have that
\begin{equation} \label{eq:k_ivalues}
  \{k_{i,j}: 1\leq j\leq n_i\}\mbox{ takes $n_i$ different values}.
\end{equation}
In particular $\mu<\nu$ and the sequence $(k_{i,j} : 1\leq i\leq l,1\leq j\leq n_i)$ takes the value $\mu$ (resp. $\nu$) at most $l$ times.

We claim that $H_{[\mu \uparrow \nu-1]}\times C$ and $H_{[\mu+1 \uparrow \nu]}\times C$ naturally embed in $(H_{[\mu\uparrow \nu]}\times C)/\normgen{f(w)}$, and under the hypothesis of (ii), these groups have solvable membership problem in $(H_{[\mu\uparrow \nu]}\times C)/\normgen{f(w)}$. Therefore we can argue as in Case 1: we have an HNN-extension 
\begin{align}
\widehat{G}/\normgen{w} &\cong \widehat{G}/\normgen{f(w)}\nonumber\\
&= \left\langle \ \frac{H_{[\mu\uparrow \nu]}\times C}{\normgen{f(w)}}, \alpha \ \vline\ \alpha(H_{[\mu \uparrow \nu-1]}\times C)\alpha^{-1}= H_{[\mu+1 \uparrow \nu]}\times C\right\rangle \label{eq:HNN2}.
\end{align}
\noindent and $f(\widehat{A}\times C)$ naturally embeds in $\widehat{G}/\normgen{f(w)}$ and hence $\widehat{A}\times C$ naturally embeds $\widehat{G}/\normgen{w}$. Using the amalgamated free product splitting of $\widehat{A}$ it follows that $A\times C$ embeds $G/\normgen{w}$.

Moreover, if $A$ and $B$ are algorithmically locally indicable, then under the conclusions of the claim, the normal forms in the HNN-extension \eqref{eq:HNN2} are computable and since $\widehat{\phi_{\widehat{A}}}$ is computable, we can conclude arguing as in the previous case, that the membership problem for $f(\widehat{A}\times C)$ in $\widehat{G}/\normgen{f(w)}$ is solvable. Hence the membership problem for $\widehat{A}\times C$ in $\widehat{G}/\normgen{w}$ is solvable.  Since $\gen{a}$ has solvable membership problem in $A$, the normal forms in the free product with amalgamation $A*_{\gen{a=\alpha^n}}\gen{\alpha |\;}$ can be effectively computed and the membership problem for $A\times C$ in $G/\normgen{w}$ is solvable, as required.

We now need to proof the claim. We show only that $H_{[\mu \uparrow \nu-1]}$  naturally embeds in $(H_{[\mu\uparrow \nu]}\times C)/\normgen{f(w)}$ and that under the hypothesis of (ii), this group has solvable membership problem in $(H_{[\mu\uparrow \nu]}\times C)/\normgen{f(w)}$. The corresponding proof for  $H_{[\mu+1 \uparrow \nu]}$ is similar.  

By \eqref{eq:k_ivalues} and the remark below, the sequence $(k_{i,j} : 1\leq i\leq l,1\leq j\leq n_i)$ takes the values $\nu$ at most $l$ times. Hence $f(w)$ has length at most $l$ in $(H_{[\mu \uparrow \nu-1]}\times C)*_C(B^{\nu}\times C)$. If the length is strictly smaller than $l$, then  by induction hypothesis $H_{[\mu \uparrow \nu-1]}\times C$  naturally embeds in $(H_{[\mu\uparrow \nu]}\times C)/\normgen{f(w)}$. Under the extra hypothesis of (ii) along with \eqref{eq:Hloq}, $H_{[\mu \uparrow \nu-1]}$ is algorithmically locally indicable and by induction, $H_{[\mu \uparrow \nu-1]}\times C$  has solvable membership problem in $(H_{[\mu\uparrow \nu]}\times C)/\normgen{f(w)}$.

The remaining case is when $f(w)$ has length $l$ in the free product with amalgamation $(H_{[\mu\uparrow \nu-1]}\times C)*_{C}(B^{\nu}\times C)$ and hence $f(w)=h_1 (\lsup{\nu}{b})^{\eps_1}\dots h_l (\lsup{\nu}{b})^{\eps_l}c$, where $\eps_i=\pm 1$. As we may change $f(w)$ to a conjugate or to its inverse, without loss of generality, we can assume that $\eps_1=1$. We apply an automorphism $\phi$ to $(H_{[\mu\uparrow \nu-1]}\times C)*_{C}(B^{\nu}\times C)$ that fixes $H_{[\mu\uparrow \nu-1]}\times C$ and sends $\lsup{\nu}{b}$ to $\lsup{\nu}{b} h_2^{-1}$. Then $\phi(f(w))$ has length less than $l$ in $(H_{[\mu\uparrow \nu-1]}\times C)*_{C}(B^{\nu}\times C)$ and by induction hypothesis, $H_{[\mu \uparrow \nu-1]}\times C=\phi(H_{[\mu \uparrow \nu-1]}\times C)$  naturally embeds in $(H_{[\mu\uparrow \nu]}\times C)/\normgen{\phi(f(w))}$ and hence, in  $(H_{[\mu\uparrow \nu]}\times C)/\normgen{f(w)}$. Also, under the extra hypothesis of (ii) along with \eqref{eq:Hloq}, $H_{[\mu \uparrow \nu-1]}$ is algorithmically 
locally indicable and by induction, $H_{[\mu \uparrow \nu-1]}\times C$  has solvable membership problem in $(H_{[\mu\uparrow \nu]}\times C)/\normgen{f(w)}$.

\textbf{Subcase 2.2:} $B$ arbitrary.

We first prove (i). The map $\widehat{\phi_B}$ extends to a homomorphism $\pi$ from $(A\times C)*_C (B\times C)$ to $(A\times C)*_C( \gen{b}\times C)$, for some $b \in B$. Then $\pi(w)$, which has length at most $l$, is not conjugate to an element of $A\times C$ or $\gen{b} \times C$. We apply Case 1 and  Subcase 2.1 to conclude that $A\times C$ naturally embeds in $((A\times C)*_C( \gen{b}\times C))/\normgen{\pi(w)}$, and hence $A\times C$ naturally embeds in  $G/\normgen{w}$. This completes the proof of (i) and from now on, we may assume that (i) holds.

Recall that $\widehat{\phi_A}(w)=m\neq 0$ and $\widehat{\phi_B}(w)=n\neq 0.$ In order to show (ii), let $\psi\colon A*B\to \Z\cong\gen{u|\;}$ be defined by $\psi(a)=n\phi_A(a)$ for $a\in A$, and $\psi(b)=-m\phi_B(b)$, for $b\in B$. Setting $K=\ker \psi$, we can regard $A*B$ as a semidirect product $K\rtimes \gen{u|\;}$. In particular, we can rewrite $w$ as \begin{equation}\label{eq:winK}w=\tilde{a_1}u^{nm_1} \tilde{b_1}u^{-mn_1}\cdots  \tilde{a_l}u^{nm_l} \tilde{b_l}u^{-mn_l}zc,\end{equation}
where $\tilde{a_i},\tilde{b_i}\in K$ and $m_i=\phi_A(a_i)$, $n_i=\phi_B(b_i)$ for $i=1,\dots, l$. 
 Since $\psi$ is computable and normal forms are easily constructed in free products, we have that
\begin{equation}\label{eq:semidirect}
\text{for any given $g\in K\rtimes \gen{u|\;}$, we can decide whether or not $g\in A$.}
\end{equation}
Consider the group $G_1= (A*B*\gen{z|\;})\times C$ and the extension $\widehat{\psi}$ of $\psi$ defined by $\widehat{\psi}:G_1\to \gen{u|\;},$ $\widehat{\psi}(g)=\psi(g)$ if $g\in A\cup B$ and $\psi(g)=1$ if $g\in C\cup \{z\}$.

We have an embedding $\iota$ of $G=(A*B)\times C$ into $G_1$ via $a\mapsto a$, $b\mapsto z^{-1}bz$, $c\mapsto c$.
Then, from \eqref{eq:winK} $$\iota(w)= \tilde{a_1}u^{nm_1} z^{-1}\tilde{b_1}u^{-mn_1}z\cdots  \tilde{a_l}u^{nm_l} z^{-1}\tilde{b_l}u^{-mn_l}zc.$$
 The total sum of the exponents of $u$ appearing in $\iota(w)$ is $n\sum m_i-m\sum {n_i}=nm-mn=0$. So, we can rewrite $\iota(w)$ as 
$$\iota(w)= k_1(u^{-t_1}z^{-1}u^{t_1})k_2(u^{-t_2}zu^{t_2}) \cdots k_{2l-1} (u^{-t_{2l-1}}z^{-1}u^{t_{2l-1}})k_{2l}(u^{-t_{2l}}zu^{t_{2l}})c$$
where $k_i\in K$ for $i=1,\dots, 2l$ and $t_{2j}=-n\sum_{i=1}^j m_i+m\sum_{i=1}^j n_{i}$ and $t_{2j-1}=-n\sum_{i=1}^j m_i+m\sum_{i=1}^{j-1} n_{i}$ for $j=1,\dots, l$. Observe that $t_{2l}=0$ and, since at least two $n_i$'s and two $m_i$'s are non-zero, the sequence $(t_i:i=1,\dots, 2l)$  takes at least one non-zero value. Let $\mu=\min \{ t_i : i=1,\dots, 2l\}$ and $\nu=\max\{t_i:i=1,\dots 2l\}$; by the previous observation $\mu<\nu$.

For $p\leq q,$ let $$H_{[p\uparrow q]}=K*\gen{u^{{-p}}zu^{{p}}}*\dots *\gen{u^{{-q}}zu^{q}}.$$
We can express $G_1= (A*B*\gen{z|\;})\times C$ as the HNN extension
\begin{equation}
G_1\cong \gen{ H_{[\mu\uparrow \nu]}\times C, u \,|\, u^{-1}(H_{[\mu \uparrow \nu-1]} \times C)u= H_{[\mu+1 \uparrow \nu]}\times C}
\end{equation}
and by (i), the groups $H_{[\mu \uparrow \nu-1]} \times C$ and  $H_{[\mu+1 \uparrow \nu]}\times C$ naturally embed in $(H_{[\mu\uparrow \nu]}\times C)/\normgen{\iota(w)}$. Hence, $G_1/\normgen{\iota(w)}$ may be written as an HNN extension as follows. 
\begin{equation}\label{eq:HNN22}
G_1/\normgen{\iota(w)}\cong \gen{ (H_{[\mu\uparrow \nu]}\times C)/\normgen{\iota(w)}, u \,|\, u^{-1}(H_{[\mu \uparrow \nu-1]} \times C)u= H_{[\mu+1 \uparrow \nu]}\times C}.
\end{equation}

Since there are at most $2l$ occurrences of $z$ or its inverse in $\iota(w)$, both $\mu$ and $\nu$  can not occur in the sequence $(t_i:i=1,\dots, 2l)$ more that $l$ times. Say that $\nu$ occurs at  most $l$ times. This implies that $\iota(w)$ is a word of length at most $l$ in $(H_{[\mu\uparrow \nu-1]}*\gen{u^{-\nu}zu^{\nu}})\times C$. Observe that $K\leqslant A*B$ is algorithmically locally indicable and so is the free product $H_{[\mu \uparrow \nu-1]}$. Then by Case 2.1. 
\begin{equation}\label{eq:member} 
\text{$H_{[\mu  \uparrow \nu-1]}\times C$ has solvable membership problem in $\dfrac{(H_{[\mu\uparrow \nu-1]}\times C)*_C(C\times \gen{u^{-\nu}zu^{\nu}})}{\normgen{\iota(w)}}$,}
\end{equation} 
and using the conjugation by $u$, we deduce that the membership problem for $H_{[\mu+1  \uparrow \nu]}\times C$ in $(H_{[\mu+1\uparrow \nu]}\times C)*_C(C\times \gen{u^{-\mu}zu^{\mu}})/\normgen{\iota(w)}$ is also solvable. The normal forms for the HNN extension \eqref{eq:HNN22} can be effectively computed.

Now, we can solve the membership problem for $A\times C$ in $G/\normgen{w}$ by solving the membership problem for $\iota (A\times C)=A\times C$ in $G_1/\normgen{\iota(w)}$, and we can do it as follows. Given $g\in G_1/\normgen{\iota(w)}$,  we can compute $\widehat{\psi}(g)$ and then, if $\widehat{\psi}(g)\notin \widehat{\psi}(A)$ we conclude that $g\notin A\times C$. Suppose that $\widehat{\psi}(g)\in \widehat{\psi}(A)$, then we can compute $a\in A$ such that $\widehat{\psi}(a)=\widehat{\psi}(g)$ and then $g\in A\times C$ if and only if $ga^{-1}\in (K\cap A)\times C$. Since normal forms in the HNN extension \eqref{eq:HNN22} can be effectively computed, we can decide whether or not $ga^{-1}$ belongs to the base group. By \eqref{eq:member}, we can decide whether or not $ga^{-1}\in H_{[\mu \uparrow \nu-1]}\times C$. Using normal forms for free products, we can decide whether or not $ga^{-1}\in K\times C$. Finally, by \eqref{eq:semidirect} we can decide whether or not $ga^{-1}\in (K\cap A)\times C$. This completes the proof of 
the theorem.
\end{proof}

\section{Direct products} \label{sec:DP}

\begin{Lem}\label{Lem:exten}
Let $A$, $B$ be groups and $a\in A$ and $b\in B$. Let $\pi$ be the natural quotient map $$\pi\colon \dfrac{A\times B}{\normgen{(a,b)}}\to \frac{A}{\gen{a^A}} \times \frac{B}{\gen{b^B}}.$$
Then $\ker{\pi}$ is abelian of finite rank.
\end{Lem}
\begin{proof}

Observe that $\gen{(a,b)^{A\times B}}$ is a normal subgroup of $\gen{a^A}\times\gen{b^B}$.  Therefore, the group $G=A\times B/\normgen{(a,b)}$ fits into the following exact sequence.
\begin{equation}\label{eq:extWP}1 \rightarrow \frac{\gen{a^A} \times \gen{b^B}}{\gen{(a,b)^{A\times B}}} \rightarrow G \stackrel{\pi}{\rightarrow} \frac{A}{\gen{a^A}} \times \frac{B}{\gen{b^B}} \rightarrow 1\end{equation}
where $\pi$ is the natural quotient map.

Let $a_1\in A$; then $(a,b)^{(a_1,1)}=(a^{a_1},b)$ and hence,  
$$(a^{a_1},b)(a,b)^{-1}=(a_1^{-1}aa_1a^{-1},1)=([{a_1}^{-1},a],1)\in \gen{(a,b)^{A\times B}}.$$
Therefore, $([A,a],1)\subseteq \gen{(a,b)^{A\times B}}$  and similarly, $(1,[B,b])\subseteq \gen{(a,b)^{A\times B}}.$ There is a natural epimorphism $\phi$,  $$\frac{\gen{a^A}}{\gen{[a,A]}}\times\frac{\gen{b^B}}{\gen{[b,B]}}\stackrel{\phi}{\to} \frac{\gen{a^A} \times \gen{b^B}}{\gen{(a,b)^{A\times B}}} =\ker (\pi).$$ 

Note that $\gen{a^A}/\gen{[a,A]}$ and $\gen{b^B}/\gen{[b,B]}$ are cyclic and hence, $\ker (\pi)$ is abelian of  rank at most 2. 

%

\end{proof}

\begin{ThmD} Let $A$ and $B$ be two recursively presented groups and $a\in A$ and $b\in B$ such that the word problem is solvable in $A$, $A/\normgen{a}$, $B$ and $B/\normgen{b}$. Then the word problem is solvable in  $G=(A \times B)/\normgen{(a,b)}$. 
\end{ThmD} 
\begin{proof}
By Lemma \ref{Lem:exten}, we conclude that the recursively presented group $G$ is an extension of two groups in each of which, the word problem is solvable. Therefore, $G$ has solvable word problem.

\end{proof}

\section{Proof of  Theorems A and B}\label{sec:ThmB}

In order to prove Theorem B we need the following observation.
\begin{Lem}\label{lem:polyfg}
Any quotient of a polycyclic group is polycyclic. 
In particular, any quotient of a polycyclic group has solvable word problem.
\end{Lem}

\begin{ThmB} Let $\Gamma$ be a starred graph and $\mathfrak{G}=\{G_v\mid v\in V\ga\}$ be a family of poly-(infinite cyclic) groups.  Let $g\in G\coloneqq \ga \mathfrak{G}$. Then,  the word problem  of the  one-relator quotient $G/{\normgen{g}}$ is solvable.
\end{ThmB} 
\begin{proof}
 We argue by induction on $|V|$. If $|V|=1$ the result holds by Lemma \ref{lem:polyfg}. We suppose $|V|\geq 2$ and that the result holds for starred graphs with fewer vertices.

If $\ga$ is connected, then there exists a nodal vertex $v$, and $G=G_v\times G_{V-\{v\}}$. By induction hypothesis all one relator quotients of $G_{V-\{v\}}$ and $G_v$ have solvable word problem. Also $G_{V-\{v\}}$ and $G_v$ are finitely presented. Therefore, by Proposition 1, $G/\normgen{g}$ has solvable word problem.

If $\ga$ is disconnected, then $\ga$ is the union of two disjoint components $\ga_1$ and $\ga_2$. In this case, $G=G_{\ga_1}*G_{\ga_2}.$ If $g$ is conjugate to an element of $G_{\ga_1}$ then, $G/\normgen{g}=G_{\ga_1}/\normgen{g}*G_{\ga_2}$. By induction, $G_{\ga_1}/\normgen{g}$ has solvable word problem and so, $G$, being the free product of two groups with solvable word problem, also has solvable word problem. 

The remaining case is when $g$ is not conjugate to an element of $G_{\ga_1}$ nor of $G_{\ga_2}$. But by Proposition \ref{prop:chordal}, $G_{\ga_1}$ and $G_{\ga_2}$ are algorithmically locally indicable. Therefore, we can apply the special case of Theorem C (ii) in which the amalgamated subgroup is trivial to conclude that $G_{\ga_1}$ has solvable membership problem in $G/\normgen{g}$; in particular, the word problem for $G/\normgen{g}$ is solvable.
\end{proof}

\begin{ThmA} Let $A_\Gamma$ be a starred right-angled Artin group and let $g$ be an element of $A_\Gamma$. Let $N$ be the set of nodal vertices of $V$ and $G\coloneqq A_\Gamma/{\normgen{g}}$. Then
\begin{enumerate}
\item[(i)] the word problem is solvable in $G$;
\item[(ii)] if $U\subset N$ and $g\not\in A_{\ga_U}$, then $A_{\ga_U}$ naturally embeds in $G$;
\item[(iii)] if $U$ spans a sub-star, and $g$ is not conjugate to an element of $A_{\ga_U}$ then $A_{\ga_U}$ naturally embeds in $G$.
\end{enumerate}
\end{ThmA} 
\begin{proof}
Assertion (i) follows from Theorem B. To prove (ii), we argue by induction on the number of vertices of $\ga$. Recall that $N$ is the set of nodal vertices of $\ga$ and $U$ is a subset of $N$ such that $g\notin A_{\ga_U}$. If $\ga$ is disconnected then $N$ is empty and the result holds. So we will suppose that $\ga$ is connected.

Let $g\in A_{\ga_N}$. As $A_\ga=A_{\ga_{V-N}}\times A_{\ga_N}$, the group $A_\ga/\normgen{g}$ decomposes as a direct product of $A_{\ga_{V-N}}$ and $A_{\ga_N}/\normgen{g}$. We know that $\ga_{N}$ is a complete graph and $A_{\ga_N}$ is abelian. If $g\notin A_{\ga_U}$,  then $\gen{g}\cap A_{\ga_U}=\{1\}$ and $A_{\ga_U}$ naturally embeds in $A_{\ga}/\normgen{g}$.

Suppose that $g\notin A_{\ga_N}$. Now $\ga_{V-N}$ is disconnected; let $K_1$ and $K_2$ be non-empty graphs that form a partition of $\ga_{V-N}$. If $g$ is conjugate to an element of $A_{K_1}\times A_{\ga_N}$, then by induction hypothesis, $A_{\ga_N}$ naturally embeds in $(A_{K_1}\times A_{\ga_N})/\normgen{g}$. Therefore, we can form the amalgamated free product $(A_{K_1}\times A_{\ga_N})/\normgen{g} *_{A_{\ga_N}}(A_{K_2}\times A_{\ga_N})$ and $A_{\ga_U}$ naturally embeds in it.
It remains to consider the case when $g$ is not conjugate to an element of either $A_{K_1}\times A_{\ga_N}$ or $A_{K_1}\times A_{\ga_N}$. Here, we invoke Theorem C(i) to conclude that $A_{\ga_N}$ naturally embeds in $A_\ga/\normgen{g}$ and therefore $A_{\ga_U}$ embeds in $G$ as required.

For (iii), let $K_1,\dots, K_n$ be the connected components of $\ga-N$. Then $U$ spans a sub-star of $\ga$ if $U=N\cup (\cup_{i\in I}VK_i)$, where $I$ is a proper subset of $\{1,\dots,n\}$. By hypothesis $g$ is not conjugate to an element of $A_{\ga_U}$. The case $I$ is empty was considered in (ii). So we assume that $I$ is non-empty. Let $J=\{1,\dots, n\}-I$ and $U^c=N\cup(\cup_{j\in J} VK_j)$. If $g$ is not conjugate to an element of $A_{\ga_{U^c}}$, then the result follows from Theorem C(i), since $G=(A_{\Gamma_{U-N}}\times A_{\Gamma_{N}})*_{A_{\Gamma_{N}}}(A_{\Gamma_{U^{c}-N}}\times A_{\Gamma_{N}})$.

Finally, suppose that $g$ is conjugate to an element of $A_{\ga_{U^c}}$. Since $g\notin A_{\ga_N}$, (ii) implies that $A_{\ga_N}$ naturally embeds in $A_{\ga_{U^c}}/\normgen{g}.$ Then $A_\ga/\normgen{g}\cong A_{\ga_{U^c}}/\normgen{g}*_{A_{\ga_U}}A_{\ga_U}$ and hence $A_{\ga_U}$ naturally embeds in $A_\ga/\normgen{g}$.
\end{proof}

\begin{Ex}
In this example we show that Theorem A can not be generalised to graph products of poly-(infinite-cyclic) groups.

For $i=1,2,3$, let $H_i$ be the poly-(infinite-cyclic) group $\gp{x_i,y_i}{x_iy_ix_i^{-1}=y_i^{-1}}$. Consider $(y_1,y_2)\in H_1\times H_2$. It is easy to see that $(y_1^2,1)$ and $(1,y_2^2)$ belong to $\gen{(y_1,y_2)^{H_1\times H_2}}$. However neither $H_1$ nor $H_2$ naturally embed in $(H_1\times H_2)/\normgen{(y_1,y_2)}$.

Similarly, let $\ga$ be the graph consisting of a line of length 2. We denote the vertices of $\ga$ by $\{1,2,3\}$, where $2$ represents the central vertex. Let $G$ be the graph product of the $H_i$'s over $\ga$ and set $g=y_1y_2$. The previous argument shows that $H_2$ does not embed in $G/\normgen{g}$ under the natural map. Using normal forms, we can see that $g$ is not conjugate to an element of $G_{\{2,3\}}=H_2 \times H_3$ but the natural map from $G_{\{2,3\}}$ to $G/\normgen{g}$ is not injective, since $H_2$ does not embed in $G/\normgen{g}$.
\end{Ex}

\medskip

\noindent{\textbf{{Acknowledgments}}} The authors wish to thank Ashot Minasyan for carefully reading the manuscript and for helpful discussions.


\begin{thebibliography}{21}


\bibitem{BBaumslag84}
Benjamin Baumslag,
\newblock{\em  Free products of locally indicable groups with a single relator},
\newblock Bull.\ Austral.\ Math.\ Soc.\ \textbf{29}(1984), 401--404.
\vskip-0.4cm \null
%
\bibitem{BCRS}
Gilbert Baumslag, Frank B.\ Cannonito, Derek J.\ S.\ Robinson, Dan Segal,
\newblock{\em The algorithmic theory of polycyclic-by-finite groups},
\newblock J. Algebra {\bf 142} (1991), no 1, 118--149.
\vskip-0.4cm \null
%
\bibitem{Brodskii84}
S.\ D.\ Brodski\u{\i},
\newblock{\em  Equations over groups, and groups with one defining relation},
\newblock Siberian Math.\ J.\ \textbf{25} (1984), 235--251.
\vskip-0.4cm\null 

\bibitem{Droms}
Carl Droms,
\newblock{\em Subgroups of Graphs Groups},
\newblock Journal of Algebra, \textbf{110} (2) (1987),519--522.
\vskip-0.4cm \null

\bibitem{Green}
Elisabeth R. Green,
\newblock \emph{Graph products},
\newblock Ph. D. thesis, Univ.\ of Leeds, 1990.
\vskip-0.4cm \null

\bibitem{Howie81}
James Howie,
\newblock {\em On pairs of $2$-complexes of equations over groups},
\newblock J. Reine Angew. Math.  \textbf{324} (1981), 165--174.
\vskip-0.4cm


\bibitem{Kapovich}
I. Kapovich, R. Weidmann and A.G. Myasnikov,
\newblock {\em Foldings,Graphs of Groups and the Membership Problem,}
\newblock Internat. J. Algebra Comput. {\bf 15} (2005), no. 1, 95--128. 
\vskip-0.4cm \null

\bibitem{LyndonSchupp}
R.C. Lyndon and P.E. Schupp,
\newblock {\em Combinatorial Group Theory, Reprint of the 1977 ed.},
\newblock {Classics in Mathematics,\ Springer-Verlag Berlin.\ (2001).}
\vskip-0.4cm \null

\bibitem{Mazurovskii}
V. Mazurovskii,
\newblock{\em The word problem for anomalous products of groups (Russian)},
\newblock Algebraic systems (Russian), (1991), 26-35.
\vskip-0.4cm \null

\bibitem{Mikhailova}
K.A. Mikhailova,
\newblock{\em The occurrence problem for direct products of groups (Russian)},
\newblock Mat. Sb. (N.S.), \textbf{70} (112) (1966), 241--251.
\vskip-0.4cm \null

\bibitem{Mikhailova68}
K.A. Mikhailova,
\newblock{\em The occurrence problem for free products of groups (Russian)},
\newblock Mat. Sb. (N.S.) \textbf{75} (117) (1968), 199--210.
\vskip-0.4cm \null

\bibitem{serre}
J.-P. Serre,
\newblock{\em Trees. Translated from the French original by John Stillwell.},
\newblock Springer Monographs in Mathematics. Springer-Verlag, Berlin, (2003).
\vskip-0.4cm \null


\end{thebibliography}
\end{document}